\documentclass{amsart}
\usepackage{amsthm,amsmath,amsfonts,amssymb,amscd,mathrsfs,graphicx}
\usepackage{hyperref}
\usepackage{setspace}
\usepackage{floatrow}
\newtheorem{thm}{Theorem}
\newtheorem*{thm*}{Theorem}
\newtheorem{lem}{Lemma}

\theoremstyle{definition}

\newtheorem*{exa*}{Example}
\newtheorem{exa}{Example}
\newtheorem*{rem*}{Remark}
\newtheorem*{rems*}{Remarks}

\usepackage{todonotes}
\usepackage{float}

\definecolor{pistachio}{rgb}{0.58, 0.77, 0.45}
\definecolor{eggshell}{rgb}{0.94, 0.92, 0.84}

\newcommand{\mb}{\mathbb}

\title{When are Multiples of Polygonal Numbers again Polygonal Numbers?}
\author{Jasbir S. Chahal}

\author{Michael Griffin}

\author{Nathan Priddis}


\begin{document}

\renewcommand{\thefootnote}{\fnsymbol{footnote}} 
\renewcommand{\thefootnote}{\arabic{footnote}}

\begin{abstract}
Euler showed that there are infinitely many triangular numbers that are three times other triangular numbers. In general, it is an easy consequence of the Pell equation that for a given square-free $m > 1$, the relation $\Delta = m\Delta'$ is satisfied by infinitely many pairs of triangular numbers $\Delta$, $\Delta'$. 

After recalling what is known about triangular numbers, we shall study this problem for higher polygonal numbers. Whereas there are always infinitely many triangular numbers which are fixed multiples of other triangular numbers, we give an example that this is false  for higher polygonal numbers. However, as we will show, if there is one such solution, there are infinitely many. We will give conditions which conjecturally assure the existence of a solution. But due to the erratic behavior of the fundamental unit of $\mathbb Q(\sqrt{m})$, finding such a solution is exceedingly difficult.
Finally, we also show in this paper that, given $m > n > 1$ with obvious exceptions, the system of simultaneous relations $P = mP'$, $P = nP''$ has only finitely many possibilities not just for triangular numbers, but for triplets $P$, $P'$, $P''$ of polygonal numbers, and give examples of such solutions.
\end{abstract}

\maketitle

\section{Introduction}

For an integer $\ell \ge 3$, the $r$-th \emph{polygonal number} or $r$-th $\ell$-\emph{gonal number} $P(\ell, r)$, $r \ge 1$, is defined by
 \[P(\ell, r) = \frac{(\ell-2)r^2-(\ell-4)r}2.\]

It is the total number of dots that are used to enlarge recursively a single dot, representing the first $\ell$-gonal number, to a regular $\ell$-gon with each side having $r$ dots. In the process the first dot remains the only common vertex of successive $\ell$-gons with each side having $r$ dots (cf. Figure \ref{fig1}). 

\begin{figure}[H]
\setlength{\unitlength}{1mm}
\begin{picture}(80,45)
 \put(0,5){\includegraphics{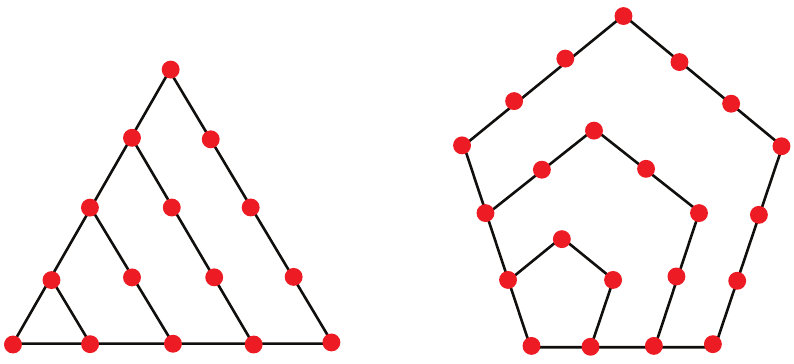}}
 \put(10,0){$P(3, 5)=15$}
 \put(52,0){$P(5,4) = 22$}
\end{picture}
 \caption{Representation of polygonal numbers for $\ell=3$ and $\ell=5$.}\label{fig1}
\end{figure}

Thus
 \[P(3, r) = \frac{r(r+1)}2 = 1, 3, 6, 10, 15, \ldots\]
are the \emph{triangular numbers},
 \[P(4, r) = r^2 = 1, 4, 9, 16, 25, \ldots\]
are the (perfect) squares, whereas
 \[P(5, r) = \frac{r(3r-1)}2 = 1, 5, 12, 22, 35, \ldots\]
are the \emph{pentagonal numbers}, and so on.

These numbers, especially the triangular and the pentagonal numbers, have been studied since antiquity (cf. \cite{dick}). In 1638 Fermat stated without proof his Polygonal Number Theorem: \textit{Every positive integer is a sum of at most $\ell$ $\ell$-gonal numbers}. The case $\ell=4$ is a well-known theorem of Lagrange (\textit{every positive integer is a sum of at most four squares}), proved by Jacobi and independently by Lagrange in 1772. Gauss discovered it for $\ell =3$ as noted in his diary on 10 July 1796. Fermat's Polygonal Number Theorem for all $\ell$ was finally proved by Cauchy in 1813.

Euler determined all triangular numbers that are three times another triangular number. Given an integer $m>1$, it is interesting to study which triangular numbers are $m$ times another triangular number. In this paper, we recall the answer to this question (see \cite{chadsouz}) and determine, more generally, which $\ell$-gonal numbers are $m$ times another $\ell$-gonal number. 

It is an easy consequence of the Pell equation, as we recall, that for any given square-free $m > 1$, the relation $\Delta = m\Delta'$ is satisfied by infinitely many pairs of triangular numbers $\Delta$ and $\Delta'$. However, this is not true for higher $\ell$-gonal numbers. We will give an example such that $P=mP'$ has no solution in $\ell$-gonal numbers (see Example~\ref{eg:nonexample}). Indeed, we will state a condition that, when satisfied, determines infinitely many solutions. We conjecture that no solutions exist when the stated condition fails. 

In contrast, we also prove that, given $m > n > 1$, the simultaneous relations
 \[\begin{cases}
 P = mP'\\
 P = nP''\end{cases}\]
have only finitely many possibilities, not just for triangular numbers but for triplets $P$, $P'$, $P''$ of $\ell$-gonal numbers for all $\ell \ge 3$ except when $\ell = 4$ and $m$ and $n$ are both perfect squares. Table~\ref{table:simultaneousSols} gives several non-trivial examples of this phenomenon for small $\ell$, $m$, and $n$.



\section{The Pell equation and polygonal numbers}

The \emph{diophantine equations} (after Diophantus of Alexandria, a third century AD Greek mathematician and the author of a series of books called \textit{Arithmetica}) are polynomial equations with integer coefficients. The solutions sought are integers or sometimes rationals. One of the most famous diophantine equations is the so--called Pell equation. 
 \begin{equation}\label{e:Pell}
x^2 - my^2 = 1.
 \end{equation}
Because the square factors of $m$ can be absorbed by $y^2$, the integer $m\ne 0$ is generally assumed to be square-free. We will not assume this in this article. We only assume that $m$ is not a perfect square. 

This equation is named after the Englishman John Pell (1611--1685), who contributed nothing to its study.  Due to an error by Euler, the name Pell was given to this diophantine equation, and, according to Weil (\cite[p.~174]{weil}), ``traditional designation [of equation \eqref{e:Pell}] as Pell's equation is unambiguous and convenient.'' 

The Pell equation has a very long and rich history going at least as far back as Archimedes (cf.~\cite{len}, \cite{nels}, \cite{weil}). His complicated ``cattle problem'' reduces to the Pell equation \eqref{e:Pell} with $m=4,729,494$. Its first nontrivial solution (i.e.~with $y\ne 0$) was found by computer search (cf.~\cite{nels}) not too long ago (in 1980).

If we write \eqref{e:Pell} as
\begin{equation*}
(x + y\sqrt m)(x - y\sqrt m) = 1,
\end{equation*}
and assume that $m$ is square--free, then solving the Pell equation is equivalent to finding the units of $\mb Z[\sqrt m]$, a special case of the Dirichlet unit theorem for the ring of integers of a number field. (Actually, this is true only if $m\equiv 0,2,3\pmod 4$, if $m\equiv 1 \pmod 4$, it needs slight modification (cf.~\cite[pp.~78--89]{jsc}) but is not an issue here.)

The solution $(x, y)$ with smallest $x=x_1 > 0$ (hence the smallest $y=y_1 > 0$) is called the \emph{fundamental solution} of \eqref{e:Pell}. Correspondingly, $x_1 + y_1\sqrt m$ is called the \emph{fundamental unit} of $\mb Z[\sqrt m]$ or by abuse of language, that of the real quadratic field $\mb Q(\sqrt m)$.

Finding the fundamental solution of \eqref{e:Pell} is still a very challenging problem, with no apparent pattern to it (cf. \cite{var, weil}). The most obvious way of finding it is to put $y = 1, 2, 3,\ldots$ in $1 + my^2$ and stop as soon as it becomes a perfect square. For example, if $m=2$, $x_1 = 3$, $y_1 = 2$ and if $m=3$, $x_1=2$, $y_1=1$. Likewise, if $m=11$, $x_1=10$, $y_1=3$. Already for $m=13$, this method is not very efficient. Table~\ref{tab:mGenerators} shows the erratic behavior of the fundamental solution for certain integers.  

\begin{table}[t]\label{tab:mGenerators}
	\begin{tabular}{c|c||c|c||c|c}
		$m$ & $(x,y)$ & $m$ & $(x,y)$ & $m$ & $(x,y)$\\
		\hline
		2 & (3,2) & 10 & (19,6) & 46 & (24335,3588)\\
		3 & (2,1) & 11 & (10,3) & 47 & (48,7)\\
		5 & (9,4) & 12 & (7,2) & 94 & (2143295,221064)\\
		6 & (5,2) & 13 & (649,180) & 95 & (39,4)\\
		7 & (8,3) & 30 & (11,2) &&\\
		8 & (3,1) & 31 & (1520, 273) &&
\end{tabular}
	\caption{Fundamental solution to the Pell equation for small $m$}
\end{table}

Brahmagupta (598--670 AD), the first to solve the quadratic equation $ax^2 + bx + c = 0$ the way we do it today, Jayadeva (9th century AD), and Bhaskara II (1114--1185 AD) were the first to study the Pell equation systematically. They showed that (for $m > 1$) infinitely many solutions can be obtained from a given nontrivial solution by using the so-called Brahmagupta identity (the product of two numbers of the form $a^2 + mb^2$ is itself a number of this form). The Fibonacci identity (the set of the sums of two squares is closed under multiplication), which already occurs in Diophantus' \textit{Arithmetica}, is a special case of the Brahmagupta identity. 

This school of Indian mathematicians devised a method (cf.~\cite[pp.~25--38]{var}) they called chakravala (or circular), which they used to solve \eqref{e:Pell} for many values of $m$. For example, for $m=61$ they found the smallest solution $x = 1766319049$, $y = 226153980$. The efficiency and the simplicity of this method impressed even Andr\'e Weil, one of the greatest number theorists of the 20th century, who wrote (cf.~\cite[p.~24]{weil}), ``to have developed \textit{chakravala} and to have it applied successfully to such difficult numerical cases as $N=61$ or $N=67$ had been no mean achievement.'' (His $N$ is our $m$.)

The Indians assumed that a nontrivial solution always exists. Lagrange was the first to prove the existence of a nontrivial solution for all $m > 1$. We now have several proofs of this fact based on different disciplines of number theory. The most common (cf.~\cite[pp.~336--356]{nzm}) given in first courses in number theory uses the techniques from diophantine approximations and continued fractions (cf. \cite{nagell}). 

Although by Lagrange's theorem we know that sooner or later $1+my^2$ becomes a perfect square, the random size of $x_1$, $y_1$ makes it very difficult to predict how far we must go for arbitrarily given $m$. For example, if $m=63$, $x_1 =8$, $y_1 = 1$, whereas, for $m=61$, in the smallest solution $(x_1, y_1)$ above, $y_1$ runs into nine (decimal) digits. Although there are better methods available---such as the method of continued fractions---there is still no last word on solving the Pell equation, or equivalently, on finding the fundamental unit of a real quadratic field (cf.~\cite{len}, \cite{will}).

It is easy to check that the set $G$ of integer solutions of \eqref{e:Pell} with $x>0$ forms an abelian group under the binary operation
\begin{equation}\label{e:grouplaw}
(x, y) \ast (x', y') = (xx' + myy', xy' + x'y)
\end{equation}
with $(1, 0)$ as its identity and $(x, -y)$ the inverse of $(x,y)$. The group law \eqref{e:grouplaw} is suggested by multiplying $(x + y\sqrt m)(x'+y'\sqrt m) = X + Y\sqrt m$ and comparing the rational (respectively irrational) parts in this equation. By the Dirichlet theorem, $G$ is a cyclic group generated by the fundamental solution $g = (x_1, y_1)$. 

If $(x_n, y_n) = g^n = \underbrace{g \ast \cdots \ast g}_{n\, {\rm times}}$, it is easy to see that the integers $y_1 < y_2 < y_3 < \cdots$ increase quadratically and soon $y_n$ is so large that $1/y_n^2$ is almost zero. Thus $(x_n/y_n)^2$ is almost equal to $m$, giving a very good rational approximation to $\sqrt m$.

\begin{exa*}
If $m=2$, $g = (x_1, y_1) = (3, 2)$ and it takes less than a minute to compute by hand and see that $(x_5, y_5) = (3363,2378)$. It is amazing that the approximation $x_5/y_5 = 3363/2378$ to $\sqrt 2$ is the same as given by a handheld calculator.
\end{exa*}

\subsection{Triangular numbers}
As an application of the above discussion of the Pell equation, we recall from \cite{chadsouz} the following generalization of Euler's work regarding which triangular numbers are three times some other triangular numbers.

 \begin{thm}\label{thm1}
For a non-square integer $m>1$, the equation
 \begin{equation}\label{e:triangle}
\Delta' = m\Delta
 \end{equation}
is satisfied by infinitely many triangular numbers $\Delta$, $\Delta'$.
 \end{thm}
 
 \begin{proof}
We first remark that if the \emph{norm form equation}
 \begin{equation}\label{e:normalform}
(x + y\sqrt m) (x - y\sqrt m) = d
 \end{equation}
has a solution $(x, y)$, infinitely many solutions are obtained by multiplying $\alpha = x + y\sqrt m$ with the units of the form $x' + y'\sqrt m$ in $\mb Z[\sqrt{m}]$. Equation \eqref{e:normalform} is called a norm form equation because the quadratic form $x^2 - my^2$ is the norm $N(\alpha) = \alpha\bar\alpha$ of $\alpha = x + y\sqrt m$ in $\mb Z[\sqrt m]$, $\bar\alpha = x - y\sqrt m$ being the \emph{conjugate} of $\alpha$.

Since the triangular numbers are the numbers $\Delta = \frac{r(r+1)}2, r=1, 2, 3, \ldots$, the equation \eqref{e:triangle} is the same as
 \[s(s+1) = mr(r+1)\]
which again on completing squares is the same as
 \begin{equation*}
(2s+1)^2 - m(2r+1)^2 = 1-m.
 \end{equation*}
So we have to find infinitely many solutions of
 \begin{equation}\label{e:Pell1-m}
X^2 - mY^2 = 1-m
 \end{equation}
with $X$, $Y > 0$ both odd. One such solution, namely $(1, 1)$, is obvious; this, however, does not give us a solution to \eqref{e:triangle}. But we can use this solution to find others. 

If $(x_n, y_n)$ is a solution of the Pell equation \eqref{e:Pell}, then it can be checked that $x_n$, $y_n$ are of opposite parity (i.e.~one odd the other even). Moreover, in case $m$ is even, it is $x_n$ that is odd. Therefore, $(X_n, Y_n) = (x_n, y_n) \ast (1, 1) = (x_n + my_n, x_n + y_n)$ is a solution of \eqref{e:Pell1-m} with both $X_n$, $Y_n$ odd. Thus solutions of \eqref{e:triangle} can be obtained from $(X_n, Y_n)$ with $n > 0$.
 \end{proof}

\begin{rem*}There is another solution to \eqref{e:Pell1-m} that will yield solutions to \eqref{e:triangle} different from those mentioned in the proof, namely $(-1,1)$ as we will see later. In fact for certain $m$ (depending on the class number of $m$), one can show that these are the only solutions to \eqref{e:triangle}. 
\end{rem*}

Suppose $a > b \ge 1$ are square-free integers which are mutually coprime. A slight modification of the above argument  with $m=ab$ in the Pell equation \eqref{e:Pell} yields the following stronger result.  We omit its proof.
 \begin{thm}\label{thm2}
The equation
 \begin{equation*}
a\Delta = b\Delta'
 \end{equation*}
has infinitely many solutions in triangular numbers $\Delta$, $\Delta'$.
 \end{thm}

\subsection{Generalization to $\ell$-gonal numbers}
Equation \eqref{e:triangle} becomes more interesting if $\Delta$, $\Delta'$ in it are replaced by higher polygonal numbers resulting in the equation
 \begin{equation}\label{e:polygnos}
P(\ell, r) = mP(\ell, s)
 \end{equation}
with $\ell \ge 5$. In this more general setting, the existence of solutions of \eqref{e:polygnos} depends on the existence of solutions of the Pell equation \eqref{e:Pell} satisfying certain congruence equations, which may or may not happen. In fact, whether the conditions are satisfied or not seems to happen randomly (cf. \cite{chaprid, prid}). Before we describe these conditions, we need the following preparatory lemma. 

\begin{lem}\label{lem:my-x}
If $(x,y)$ is a solution to \eqref{e:Pell} with $x>0$ and $y>0$, then $x>y$ and $my>x$. 
\end{lem}

\begin{proof}
Since $x^2=1+my^2$ and $m\geq 2$ it is clear that $x>y$. Furthermore, $m^2y^2+1>my^2+1=x^2$. Hence $m^2y^2> x^2$. 
\end{proof}

For every solution $(x,y)$ to \eqref{e:Pell}  with $x>0$ and $y>0$, we also have solutions $(\pm x,\pm y)$ in any combination of positive and negative. If we compose these solutions with $(1,1)$ we obtain 
\[
(1,1)\ast (\pm x,\pm y)=(\pm x \pm my, \pm x\pm y).
\] 
From the lemma, we see that only $(x+my,x+y)$ has both entries positive. We can do something similar with $(-1,1)$ and we see that of the four possibilities only 
\[
(-1,1)\ast ( x, y)=(-x+ my, x - y)
\]
has both entries positive. We could do something similar with $(-1,-1)$ and $(1,-1)$, but these are simply the negatives of what we have already done. 

\begin{thm}\label{t:polygnrs}
Given $m > 1$ not a perfect square and $\ell \ge 5$ there exists a solution to \eqref{e:polygnos} if there exists a solution to \eqref{e:Pell} with 
\begin{equation}\label{e:xy-1}
 \left.\begin{array}{l}
x+my \equiv -1 \pmod q\\[1mm]
x+y \equiv -1\pmod q
 \end{array}\right\}
\end{equation}
or
\begin{equation}\label{e:x-y-1}
\left.\begin{array}{l}
my-x \equiv -1 \pmod q\\[1mm]
x-y \equiv -1\pmod q
\end{array}\right\}.
\end{equation}
Here $q = 2\ell - 4$, $\ell - 2$, and $\ell/2 - 1$ according as $\ell\equiv 1, 3 \pmod 4$, $2\pmod 4$, and $0\pmod 4$, resp. Moreover, if one solution exists, then there are infinitely many. 
 \end{thm}


\begin{proof}
We consider only the case $\ell\equiv 1, 3\pmod 4$. The other two cases are dealt with similarly. Writing exactly what the two sides of \eqref{e:polygnos} stand for and completing squares, we obtain the generalized Pell equation
 \begin{equation}\label{e:Pell1-ml2}
X^2 - mY^2 = (1-m)(\ell - 4)^2
 \end{equation}
with
 \begin{align*}
X &= (2\ell - 4)\, r - (\ell - 4)\\
\intertext{and}
Y &= (2\ell - 4)\, s - (\ell - 4). 
 \end{align*}
Clearly \eqref{e:polygnos} has a solution if and only if the generalized Pell equation \eqref{e:Pell1-ml2} has a solution with $X$, $Y \equiv -(\ell - 4)\mod{(2\ell - 4)}$ and $X,Y>0$.

One obvious solution to \eqref{e:Pell1-ml2} is $(X, Y) = (\ell - 4, \ell-4)$. We may obtain further solutions by composing this solution with the solutions $(x,y)$ of the Pell equation \eqref{e:Pell}:  
\begin{align*}
X &= (\ell - 4)(x + my)\\
\intertext{and}
Y &= (\ell - 4)(x+y). 
 \end{align*}
Since $\ell \equiv 1, 3\pmod 4$, $\ell - 4$ and $2\ell - 4$ are coprime, so $\ell - 4$ is a unit $\pmod{2\ell - 4}$. Hence the requirement $X, Y \equiv -(\ell - 4)\pmod{2\ell -4}$ for the solutions of \eqref{e:Pell1-ml2} leads to the conditions
 \begin{align*}
x + my &\equiv -1 \pmod{2\ell - 4}\\
\intertext{and}
x + y &\equiv -1 \pmod{2\ell - 4}
 \end{align*}
for a solution $(x,y)$ to \eqref{e:Pell}.

Based on the discussion above, we consider $(-(\ell-4),\ell-4)$ also a solution to \eqref{e:Pell1-ml2}. Composing with $(x,y)$ we obtain 
\begin{align*}
X &= (\ell - 4)(-x + my)\\
\intertext{and}
Y &= (\ell - 4)(x-y). 
\end{align*}
By Lemma~\ref{lem:my-x}, $X,Y>0$. This leads to conditions \eqref{e:x-y-1}. 

As soon as we have one solution $(X_0, Y_0)$ to \eqref{e:Pell1-ml2} satisfying one of the stated congruence relations, we can find infinitely many solutions. In fact, the cyclic group of solutions to \eqref{e:Pell} modulo $q$ becomes a finite cyclic group with identity the class of solutions to \eqref{e:Pell} with $x\equiv 1\pmod q$ and $y\equiv 0\pmod q$ (cf. \cite{chaprid}). There are infinitely many such solutions. Let $(x_s,y_s)$ be such a solution with $x_s,y_s>0$. Then $(X_0,Y_0)*(x_s,y_s)$ is a solution to \eqref{e:Pell1-ml2} satisfying the required congruence relations. Each of these yields a solution to \eqref{e:polygnos}.  
\end{proof}

\begin{rems*}
\item
1. One can see from its proof, that the converse of Theorem~\ref{t:polygnrs} is also true if, for a given $m>1$, any two solutions to \eqref{t:polygnrs} are related by a unit (and therefore a solution to \eqref{e:Pell}). This is related to the class number problem. In this case, we can recover all polygonal numbers which satisfy \eqref{e:polygnos}. This is the case, for example, with $\ell=5$, $m=2$ as we will see in Example~\ref{ex:example}. 

\vspace*{.25cm}
\noindent 2. We can reformulate \eqref{e:xy-1} in the following way. 
 \begin{align*}
(m-1)x &\equiv -(m-1) \pmod{q}\\
\intertext{and}
(m-1)y &\equiv 0 \pmod{q}.
 \end{align*}
Let $d=\gcd(m-1,q)$. If $d$ is odd, we see there is only one solution to \eqref{e:xy-1}, namely $x\equiv -1 \pmod q, y\equiv 0\pmod q$. If $d$ is even, one can check, there is an additional possible solution (modulo $q$) satisfying \eqref{e:Pell}, which is $x\equiv \tfrac{q}{2}-1\pmod q$ and $y\equiv \tfrac{q}{2}\pmod q$. Only one of these is possible, and these are the only possibilities satisfying \eqref{e:xy-1}. 

\vspace*{.25cm}
\noindent 3. In \cite{chaprid, prid} the authors investigated $G_{m,q}$, the group modulo $q$ of solutions to \eqref{e:Pell}. It is a finite cyclic group and its order is denoted $g_m(q)$. Solutions to \eqref{e:xy-1} are closely related to the order of this group. Indeed from the previous remark, we see that solutions to \eqref{e:xy-1} correspond with elements of $G_{m,q}$ of order 2. Thus for example, we have no solutions to \eqref{e:xy-1} if $g_m(q)$ is odd. 
\end{rems*}

We now give some examples to show that the Equations~\eqref{e:xy-1} and \eqref{e:x-y-1} are often satisfied but not always. We emphasize we differentiate between solutions to \eqref{e:xy-1} or \eqref{e:x-y-1} from the theorem and solutions to the original problem in \eqref{e:polygnos}. When no solution to \eqref{e:xy-1} or \eqref{e:x-y-1} exists, it does not imply the same about the original problem. Verifying there are no solutions to \eqref{e:polygnos} can be much more involved. 

In the following examples, we focus mostly on solutions to \eqref{e:xy-1} and \eqref{e:x-y-1}, but in Example~\ref{eg:nonexample}, we see in fact that there are no solutions to the original problem \eqref{e:polygnos}. Also in Example~\ref{ex:example} we show that for certain choices of $\ell$ and $m$, Theorem~\ref{t:polygnrs} yields all solutions to \eqref{e:polygnos}. 

\begin{exa}\label{ex:example}
Consider $\ell=5$ (so $q=6$). These correspond to pentagonal numbers.  
For $m=2$, the fundamental solution to \eqref{e:Pell} is $(3,2)$. This does not satisfy \eqref{e:xy-1} or \eqref{e:x-y-1}. However if we look at powers of this fundamental solution, we see that the next solution is $(17,12)$, which satisfies \eqref{e:xy-1} and the third solution is $(99,70)$, which satisfies \eqref{e:x-y-1}. Hence there are infinitely many pentagonal numbers which are two times another pentagonal number. In fact, if we consider the ring $\mb Z[\sqrt{2}]$, then the pentagonal numbers which are twice another pentagonal number correspond to solutions to 
\[
x^2-2y^2=-1,
\]
which in turn are units in $\mb Z[\sqrt{2}]$. The units in $\mb Z[\sqrt{2}]$ are powers of $(1+\sqrt{2})$---which in the notation of the Pell equation we would write as $(1,1)$ as in the proof of the theorem---negatives, and conjugates of such powers. Hence, we have actually found all solutions to \eqref{e:polygnos} with $\ell=5$ and $m=2$. 

With $\ell=5$, we can do similar computations for $m=3,4,5,6,7$ finding infinitely many solutions to \eqref{e:polygnos} via the conditions in \eqref{e:xy-1}. With $m=12$, there is no solution to \eqref{e:xy-1}; however, the fundamental solution $(7,2)$ satisfies $\eqref{e:x-y-1}$, and so we obtain infinitely many solutions in this case as well. 


For $m=10,11,13$, one can easily check that there are no solutions to \eqref{e:xy-1} or \eqref{e:x-y-1} as set out in the theorem. Indeed, in the notation introduced in the Remark above, $g_{10}(5)=1$ and $g_{13}(5)=1$. That means that every solution to \eqref{e:Pell} is congruent to $(1,0)\mod q$. Furthermore, $g_{11}(5)=2$, so we only need to check the conditions on the fundamental solution. Again, this does not necessarily mean there are no solutions to \eqref{e:polygnos}, although we conjecture that there are not.


\end{exa}

\begin{exa}\label{eg:nonexample}
Consider $\ell=6$ (hence $q=4$). With $m=2$. This Pell equation has fundamental solution $(3,2)$, and as we have seen the next solution is $(17,12)$. Neither of these solutions satisfies either \eqref{e:xy-1} or \eqref{e:x-y-1}. Since $(17,12)$ is congruent to $(1,0)\pmod q$, (i.e. $g_2(4)=2$) these are the only solutions we must consider.  In fact we will now demonstrate that there are no hexagonal numbers which are twice another hexagonal number.

In this case, we look for solutions to 
\[
X^2-mY^2=(1-m)(\tfrac{\ell-4}{2})^2
\]
with $X,Y\equiv -\tfrac{\ell-4}{2}\pmod q$, which becomes 
\[
X^2-2Y^2=-1.
\]
Any solution to this equation corresponds to a unit in $\mb Z[\sqrt{2}]$ with norm $-1$---that is to say a unit with norm 1 multiplied by either $(1+\sqrt{2})$ or $(-1+\sqrt{2})$. But these are exactly the solutions to \eqref{e:Pell} multiplied by either $(1,1)$ or $(-1,1)$ as we have seen. Hence there are no solutions. 

With $\ell=6$ we can find solutions to \eqref{e:xy-1} or \eqref{e:x-y-1} with $m=3,6,7,8,11$, but not for $m=2,5,10,12,13$. Again, we conjecture that there are no solutions for these last choices of $m$. 
\end{exa}

\begin{exa}
Finally, consider $\ell=8$ (hence $q=3$). In this case we have solutions to \eqref{e:xy-1} or \eqref{e:x-y-1} when $m=2,3,5,6,7,8,12$, but not when $m=10,11,13$. (Again, this does not necessarily mean there are no solutions to \eqref{e:polygnos} for these choices of $m$.) 
\end{exa}

 \section{Elliptic Curves and Polygonal Numbers}

In what follows, by an elliptic curve we mean a diophantine equation
\begin{equation}\label{eq11}
dy^2 = x^3 + ax^2 + bx + c
\end{equation}
with $a$, $b$, $c$, $d$ in $\mb Z$, $d \ge 1$ and the right-hand side of \eqref{eq11} having no repeated root. The rational solutions of \eqref{eq11}, more precisely the set $E(\mb Q)$ of rational points on the elliptic curve $E$ defined by \eqref{eq11} together with the point $O$ at infinity is, by the Mordell-Weil theorem (cf.~\cite[p.~125]{jsc}), a finitely generated abelian group with $O$ as its identity, under the group law given by $P + Q + R = O$ if the points $P$, $Q$, $R$ on \eqref{eq11} are collinear. Thus the torsion subgroup (i.e.~the subgroup of points of finite order) of $E(\mb Q)$ is finite. By the Lutz-Nagell theorem (cf.~\cite[p.~136]{jsc}), a torsion point can have only integer coordinates. However, a point with integer coordinates need not be a torsion point. \textit{A priori} it is possible for \eqref{eq11} to have infinitely many integer solutions. But the next theorem (a special case of Siegel's theorem) from the theory of diophantine equations (cf.~\cite[p.~255]{mor}) states that this is not possible.

 \begin{thm}\label{thm3}
The elliptic curve \eqref{eq11} has only finitely many points on it with integer coordinates.
 \end{thm}
However, \eqref{eq11} may or may not have infinitely many rational solutions. In fact, both the cases are believed to occur with equal probability (cf. \cite{silverberg}). 

We now prove the following fact about polygonal numbers.

 \begin{thm}\label{thm4}
Given $\ell \ge 3$ and $m > n > 1$ ($m$ or $n$ square-free in case $\ell=4$), there are only finitely many triplets $r$, $s$ and $t$ such that
 \begin{equation}\label{e:polygdouble}
P(\ell, r) = m P(\ell, s) = n P(\ell, t).
 \end{equation}
 \end{thm}
 
 \begin{proof}
We assume $\ell \ne 4$ (the trivial case). The relation \eqref{e:polygdouble} is the same as the diophantine equations
 \begin{align*}
(\ell - 2)r^2 - (\ell - 4)r &= m((\ell - 2)s^2 - (\ell - 4)s)\\
&= n((\ell - 2)t^2 - (\ell - 4)t).
 \end{align*}
Setting $a=2(\ell-2)$ and $b=(\ell-4)$ and completing squares,  these equations become
\begin{equation}\label{e:fibonacci}
 \left.\begin{array}{l}
(ar-b)^2 - m(as-b)^2 = -(m-1)b^2\\[1mm]
(ar-b)^2 - n(at-b)^2 = -(n-1)b^2.
 \end{array}\right\}
 \end{equation}

Let
 \begin{equation*}
u = ar-b, \quad v = as-b, \quad w = at-b,
 \end{equation*}
and
 \begin{equation}\label{e:sub1}
X = m^2nv^2.
 \end{equation}
Then (\ref{e:fibonacci}) implies 
 \begin{equation}\label{e:sub2,3}
X - A = mnu^2, \quad  X- B = mn^2w^2
 \end{equation}
where $A = mn(m-1)b^2$  and $ B = mn(m-n)b^2.$

Now set $Y=m^2n^2uvw$. Multiplying \eqref{e:sub1} with both equations in \eqref{e:sub2,3}, we find
 \begin{equation}\label{e:elliptic}
Y^2 = X(X-A)(X-B).
 \end{equation}
 By our assumption on $m$, $n$ and $\ell$, \eqref{e:elliptic} is an elliptic curve with integer coefficients. So by \ref{thm3}, it has only finitely many points with integer coordinates. Consequently, \eqref{e:polygdouble} has only finitely many possibilities.
\end{proof}

\begin{exa}
Consider $\ell=5$, $m=6$ and $n=3$. We have seen that there are infinitely many solutions to each of $P(r,5)=6P(s,5)$ and $P(r,5)=3P(t,5)$. However, following the proof of Theorem~\ref{thm4}, finding simultaneous solutions to both can be reduced to finding integer pionts on the elliptic curve
\[
Y^2=X(X-90)(X-54).
\]

This curve has several integer points, however in order to correspond to an integer solution to \eqref{e:polygdouble}, we must have $m^2n=108$ divide $X$, and $m^2n^2=324$ divide $Y$. Only three such points exist: $(0,0)$, $(108, 324),$ and $(90828, 27351756)$. The first point gives  $v^2=0$, which does not yield an integer value for $s$. The second point gives $v^2=u^2=t^2=1$, which yields only the trivial solution $(r,s,t)=(0,0,0).$ The third point gives $v=\pm 29,u=\pm 71,$ and $w=\pm41$, which yields only the solution $(r,s,t)=(12,5,7)$. Hence the only simultaneous solution in pentagonal numbers to $P=6P'=3P''$ is $P=P(5,12)=210$,   $P'=P(5,5)=35$, and  $P''=P(5,7)=70.$

Table \ref{table:simultaneousSols} gives several other examples. In each case listed, the solution is unique for the given $\ell$, $m$ and $n$.
\end{exa}

\begin{table}
	\begin{tabular}{l|c|c|c|c}
		$\ell$-gonal numbers  & $P$ & $P'$ & $P''$ &$P=mP'=nP''$\\\hline 
		$\ell=3$, triangular &  6&1&3&$P=6P'=2P''$\\ 
		\hline
		& 630 & 105 & 210  &$P=6P'=3P''$\\ 
		\hline
		& 105 & 15& 21  &$P=7P'=5P''$\\ 
		\hline
		$\ell=5$, pentagonal  & 210& 35 & 70&$P=6P'=3P''$\\\hline 
		$\ell=6$, hexagonal  & 120 & 6 & 15&$P=20P'=8P''$\\\hline 
		$\ell=7$, heptagonal & 12852 &1071& 6426  &$P=12P'=2P''$\\\hline 
	\end{tabular}
\caption{Solutions to $P=mP'$ and $P=nP''$}
\label{table:simultaneousSols}
\end{table}

\begin{rem*}Relations \eqref{e:fibonacci} are a variation of equations considered by Fibonnacci in his book on squares \cite{fib}. The reader with some knowledge of algebraic geometry may also see conceptually that such a pair of equations always defines an elliptic curve. The simultaneous equations
 \begin{equation}\label{e:fib2}
\left.\begin{matrix}
x^2 + ny^2 = z^2\\
x^2 - ny^2 = t^2\end{matrix}\right\}
 \end{equation}
considered by Fibonnacci are related to the congruent number problem (cf. \cite{jsc}). The pair of equations \eqref{e:fib2} defines an elliptic curve with its Weierstrass equation
$y^2 = x^3 - n^2x.$
\end{rem*}
\section*{acknowledgment}
The second author was supported by NSF grant DMS-1502390


\vfill\eject

\end{document}